\documentclass[12pt]{amsart}
\usepackage[colorlinks=true, pdfstartview=FitV, linkcolor=blue, citecolor=blue, urlcolor=blue, breaklinks=true]{hyperref}
\usepackage{amsmath,amsfonts,amssymb,amsthm,amscd,array,comment,
euscript,mathtools,etoolbox,latexsym,stmaryrd}
\usepackage{paralist}
\usepackage[usenames]{color}
\usepackage[all]{xy}

\newtoggle{comments}
\newtoggle{details}
\newtoggle{detailsnote}

\toggletrue{comments}      

\toggletrue{details}       

\toggletrue{detailsnote}   

\newcommand\C{\mathbb{C}}
\newcommand\Z{\mathbb{Z}}

\newcommand\g{\mathfrak{g}}

\newcommand\p{\mathfrak{p}}
\newcommand\h{\mathfrak{h}}
\newcommand\n{\mathfrak{n}}

\newcommand\fb{\mathfrak{b}}
\newcommand\fsl{\mathfrak{sl}}

\newcommand\fl{\mathfrak{l}}

\newcommand\fu{\mathfrak{u}}
\newcommand\fk{\mathfrak{k}}

\newcommand\D{\Delta}

\newcommand\bU{\mathbf{U}}

\DeclareMathOperator{\im}{im} 

\DeclareMathOperator{\Aut}{Aut}

\DeclareMathOperator{\Ann}{Ann}
\DeclareMathOperator{\Supp}{Supp} 

\DeclareMathOperator{\Ind}{Ind}

\theoremstyle{plain}
\newtheorem{theo}{Theorem}[section]
\newtheorem*{theo*}{Theorem}
\newtheorem{prop}[theo]{Proposition}
\newtheorem{lem}[theo]{Lemma}
\newtheorem{cor}[theo]{Corollary}

\theoremstyle{definition}
\newtheorem{defin}[theo]{Definition}
\newtheorem*{rem*}{Remark}
\newtheorem{rem}[theo]{Remark}

\newtheorem{example}[theo]{Example}

\numberwithin{equation}{section}

\allowdisplaybreaks

\iftoggle{comments}{
\newcommand{\comments}[1]{ \begin{center} \parbox{5 in}{{\bf {\footnotesize Comments:  }}{\footnotesize \textit{#1}}} \end{center}}
}{
\newcommand{\comments}[1]{}
}

\iftoggle{details}{
\newcommand{\details}[1]{\smallskip \color{blue} \begin{footnotesize} \textbf{Details:} #1 \end{footnotesize} \color{black}}
}{
\newcommand{\details}[1]{}
}

\begin{document}
%

\title{Annihilators of simple integrable weight $\fsl(\infty)$-modules}
\author{Lucas Calixto}
\address{Department of Mathematics, Federal University of Minas Gerais}
\email{lhcalixto@ufmg.br}

\thanks{L. C. was supported by the Capes grant (88881.119190/2016-01) and by the PRPq grant (ADRC-05/2016).}




\begin{abstract}
Let $\g=\fsl(\infty)$. We compute the annihilators of a class of simple integrable weight $\g$-modules with finite-dimensional weight spaces. It is a claim of I. Dimitrov, that this class exhausts all simple integrable weight $\g$-modules with finite-dimensional weight spaces. The main feature of interest is that Dimitrov's class of modules contains non highest weight modules. Here we provide another construction for these modules, which allows to apply results of \cite{PP18} to compute such annihilators.
\end{abstract}

\maketitle \thispagestyle{empty}


%
\section{Introduction}
%

Recently the primitive ideals of the enveloping algebra of the Lie algebra $\fsl(\infty)$ has been classified by A. Petukhov and I. Penkov in \cite{PP18}. As a next step, one should like to compute the annihilators of natural classes of simple $\fsl(\infty)$-modules. The representation theory of $\fsl(\infty)$ is not as well developed as that of $\fsl(n)$. However, in the last year several classes of simple weight modules have been discussed in the literature; see \cite{Nam18, CP18, GP18}. Furthermore, in \cite{PP18}, an algorithm for computing the primitive ideals of any simple highest weight module is presented, and the problem of computing primitive ideals is open only for simple non highest weight modules. For the class of simple bounded weight modules, this problem was solved by D. Grantcharov and I. Penkov in \cite{GP18}. 

Our aim in the present note is to compute the annihilators of a class of non-bounded simple weight $\fsl(\infty)$-modules with finite-dimensional weight spaces. This class of modules has been defined by I. Dimitrov in a talk given in U.C. Berkeley \cite{D_talk}. In his talk, I. Dimitrov has claimed that these modules, up to isomorphism, are all integrable simple weight $\fsl(\infty)$-modues with finite-dimensional weight spaces. Besides simple highest weight modules, Dimitrov's class contains simple integrable modules without a highest weight. Modulo the results of \cite{PP18}, the annihilator problem reduces to the computation of the annihilators of the latter modules.
Our main result, Theorem~\ref{thm:main1}, shows that the annihilators in question equal the annihilators of certain highest weight modules, which we construct explicitly. Then our claim follows by applying the algorithm given in \cite{PP18}.

\medskip

\iftoggle{detailsnote}{
\medskip

}{}

\medskip

\paragraph{\textbf{Acknowledgements}} This paper has been written during a post-doctoral period at the Jacobs University, Bremen, under supervision of Ivan Penkov. I am grateful to Ivan Penkov for proposing the problem, for all stimulating discussions, and for valuable suggestions. I also thank the Jacobs University for the hospitality.

\medskip

\subsection*{Notation} The ground field is $\C$, and all vector spaces, algebras, and tensor products are considered to be over $\C$, unless otherwise stated. We let $\langle \ \rangle_\C$ denote the span over $\C$. We write $\C^{\Z_{>0}}$ for the set of all sequences of complex numbers. For any Lie algebra $\fk$ we let $\bU(\fk)$ denote its universal enveloping algebra.

\medskip
%
\section{Preliminaries} \label{sec:prel}
%
In what follows we set $\g=\fsl(\infty):= \varinjlim_n\fsl(n)$.   Let  $\h_n\subseteq \fsl(n)$ be the Cartan subalgebra consisting of the diagonal matrices in $\fsl(n)$, and let $\varepsilon_i\in \h_n^*$  be the functional determined by $\varepsilon_i(E_{j,j})=\delta_{i,j}$, where $E_{i,j}$ denotes the standard coordinate matrix with $1$ in the $i,j$-position and zeros elsewhere. We fix $\h = \varinjlim_n \h_n$ a Cartan subalgebra of $\g$, and we denote by $\varepsilon_i$ the vectors of $\h^*$ whose restriction to $\h_n^*$ coincide with the vectors $\varepsilon_i\in \h_n^*$. Then one can identify an element $\lambda\in \h^*$ with the formal sum $\sum \lambda_i\varepsilon_i$, or with the infinite sequence $(\lambda_1,\lambda_2, \ldots)\in \C^{\Z_{>0}}$. Let $M$ be a $\g$-module. We call $M$ an \emph{integrable module} if for all $m\in M$, $g\in \g$, we have $\dim \langle m, g\cdot m, g^2\cdot m,\ldots\rangle_\C<\infty$. We define the \emph{annihilator} of $M$ to be $\Ann M = \{u\in \bU(\g)\mid u\cdot M=0\}$. We call $M$ a \emph{weight $\g$-module} if it admits a weight space decomposition: $M=\bigoplus_{\mu\in \h^*} M_\mu$, where $M_\mu=\{m\in M\mid hm=\mu(h)m,\ \forall h\in \h^*\}$. The \emph{support} of a weight module $M$ is the set $\Supp M = \{\mu\in \h^*\mid M_\mu\neq 0\}$. All modules considered in this paper are assumed to be weight modules with finite-dimensional weight spaces, that is, $\dim M_\mu < \infty$ for every $\mu\in \Supp M$. Moreover, $M$ is called \emph{a bounded weight module} if there is a constant $c\in \Z_{>0}$ for which $\dim M_\mu < c$ for all $\mu\in \Supp M$. 

The Lie algebra $\g$ is a weight module via the adjoing action, and its weight decomposition (called root space decomposition) is given by $\g=\h\oplus \left( \bigoplus_{\alpha\in \D} \g_\alpha\right)$, where $\D=\{\varepsilon_i-\varepsilon_j\mid i,j\in \Z_{> 0}, i\neq j\}$ are the roots of $\g$.  A \emph{triangular decomposition} of $\D$ is a splitting $\D=\D^-\sqcup \D^+$ satisfying: $\alpha,\beta\in \D^+$ and $\alpha+\beta\in \D$ implies $\alpha+\beta\in \D^+$, and $\D^-=-\D^+$. Triangular decompositions of $\D$ are in bijection with linear orders on $\Z_{> 0}$:
	\[
(\Z_{> 0}, \prec)\mapsto \Delta^+=\{\varepsilon_i-\varepsilon_j\mid i\prec j,\ i,j\in \Z_{> 0}\},
	\]
and 
	\[
 \Delta^+=\{\varepsilon_i-\varepsilon_j\mid i\prec j\}\mapsto (\Z_{> 0}, \prec),\text{ where  for all }i,j\in Z_{> 0}, \  i\prec j \Leftrightarrow \varepsilon_i-\varepsilon_j\in \D^+.
	\]

Any linear order $\prec$ on $\Z_{>0}$ provides subalgebras: $\n(\prec)^\pm =\bigoplus_{i\prec j} \g_{\varepsilon_i - \varepsilon_j}$, $\fb(\prec)=\h\oplus \n(\prec)^+$. The subalgebra $\fb(\prec)$ is called the \emph{Borel subalgebra} of $\g$ associated to $\prec$. Throughout the paper we only consider Borel subalgebras of $\g$ that contains $\h$ and that can be obtained from a linear order on $\Z_{>0}$.

\begin{defin}\label{defin:Borel.subalgebras}
A Borel subalgebra $\fb(\prec)\subseteq \g$ is called a \emph{Dynkin Borel subalgebra} if and only if $(\Z_{>0},\prec)$ is isomorphic as an ordered set to $(\Z_{>0}, <)$, $(\Z_{<0}, <)$ or $(\Z, <)$. According to this a Borel subalgebra is called \emph{right-infinite}, \emph{left-infinite} and \emph{two-sided}, respectively. Dynkin Borel subalgebras are the only Borel subalgebras of $\g$ for which any positive root can be written as a finite sum of simple roots.
\end{defin}

\begin{rem}
The terminology introduced above for $\g$ also make sense for $\fsl(n)$ and it will be used freely.
\end{rem}

\section{A class of simple integrable weight modules with finite-dimensional weight spaces}\label{subsec:classification.int.mod}
The construction given in this section is due to I. Dimitrov \cite{D_talk}. It provides an explicit exhaustion for certain simple submodules of infinite tensor products that were considered previously in \cite{DP99}.

Let $\fl=\bigoplus_{k>0} \fsl(n_k)$, $n_k\in \Z_{>0}$, for all $k\in \Z_{>0}$. Choose a Borel subalgebra $\fb_{n_k}\subseteq \fsl(n_k)$, for all $k\in \Z_{>0}$, and let $\fb_{\fl}=\bigoplus_{k>0}\fb_{n_k}$. Let $L_{\fb_{n_k}}(\lambda^k)$ be a simple finite-dimensional $\fsl(n_k)$-module of highest weight $\lambda^k\in P_{\fb_{n_k}}^+$, for all $k\in \Z_{>0}$, and set $\lambda=((\lambda^1), (\lambda^2),\ldots)\in \h^*$. Consider the $\fl$-module $\bigotimes L_{\fb_\fl}(\lambda):=\bigotimes_{k>0} L_{\fb_{n_k}}(\lambda^k)$. Now, for each $k\in \Z_{>0}$, we choose a vector $v_k\in L_{\fb_{n_k}}(\lambda^k)$, and we let $\left(\bigotimes L_{\fb_\fl}(\lambda)\right)(\otimes v_k)$ denote the submodule of $\bigotimes L_{\fb_\fl}(\lambda)$ generated by the vector $\otimes_{k>0} v_k$. In other words, 
	\[
\left(\bigotimes L_{\fb_\fl}(\lambda)\right)(\otimes v_k) = \lim_{\stackrel{\longrightarrow}{\ell}} \bigotimes_{k=1}^\ell L_{\fb_{n_k}}(\lambda^k), 
	\]
where $\bigotimes_{k=1}^\ell L_{\fb_{n_k}}(\lambda^k) \hookrightarrow \bigotimes_{k=1}^{\ell+1} L_{\fb_{n_{k}}}(\lambda^k)$, $v\mapsto v\otimes v_{\ell+1}$, for all $v\in \bigotimes_{k=1}^\ell L_{\fb_{n_k}}(\lambda^k)$, and $k\in \Z_{>0}$. In particular, since $\bigotimes_{k=1}^\ell L_{\fb_{n_k}}(\lambda^k)$ is a simple integrable $\left(\bigoplus_{k=1}^\ell \fsl(n_k) \right)$-module, for every $\ell\in \Z_{>0}$, we have that $\left(\bigotimes L_{\fb_\fl}(\lambda)\right)(\otimes v_k)$ is a simple integrable $\fl$-module.

Suppose that $w_k=c_kv_k$, where $c_k\in \C$ for every $k\in \Z_{>0}$. Then
	\[
\left(\bigotimes L_{\fb_\fl}(\lambda)\right)(\otimes v_k) \cong \left(\bigotimes L_{\fb_\fl}(\lambda)\right)(\otimes w_k).
	\] 
In particular, if the weight spaces of $L_{\fb_{n_k}}(\lambda^k)$ are one-dimensional for all $k$, and, for each $k$ we choose a weight vector in $L_{\fb_{n_k}}(\lambda^k)_{\mu^k}$ with $\mu^k\in P_{\fb_{n_k}}^+$, then different choices of vectors in the same weigh spaces provide isomorphic modules. Let $L_{\fb_\fl}((\lambda), (\mu))$
denote such a module for a fixed choice of weight vectors in $L_{\fb_{n_k}}(\lambda^k)_{\mu^k}$, where $\mu = ((\mu^1), (\mu^2),\ldots)\in \h^*$ is such that $\mu^k\in \Supp L_{\fb_{n_k}}(\lambda^k)$, for all $k\in \Z_{>0}$.

Given two sequences $\gamma=(\gamma_i)\in \C^{\Z_{>0}}$ and $\eta=(\eta_i)\in \C^{\Z_{>0}}$, we write $T(\gamma)=T(\eta)$ if and only if $\gamma_i=\eta_i$ for all but finitely many indices $i\in \Z_{>0}$.

\begin{prop}
$L_{\fb_\fl}((\lambda), (\mu))\cong L_{\fb_\fl}((\lambda), (\eta))$ if and only if $T(\mu)=T(\eta)$.
\end{prop}
\begin{proof}
Since $T(\mu)=T(\eta)$, we have that the vector $\otimes v_k\in \bigotimes L_{\fb_{n_k}}(\lambda)_{\mu^k}$ lies in both modules $L_{\fb_\fl}((\lambda), (\mu))$ and $L_{\fb_\fl}((\lambda), (\eta))$. Since both modules are simple, the result follows.
\end{proof}

For a Borel subalgebra $\fb:=\fb(\prec)$ of $\g$ we define $P_\fb^+=\{\lambda:=(\lambda_i)\in \C^{\Z_{>0}} \mid \lambda_i - \lambda_j\in \Z_{\geq 0}, \text{ for all }i\prec j\}$ the set of dominant integral weights of $\g$ with respect to $\fb$. For any $\lambda\in \h^*$ we let $L_\fb(\lambda)$ denote the irreducible $\fb$-highest weight module with highest weight $\lambda$. From now on assume that $\fb$ is a Dynkin Borel subalgebra of $\g$ (see Definition~\ref{defin:Borel.subalgebras}), and we set $\fb_{n_k}:=\fb\cap \fsl(n_k)$ for all $k\in \Z_{>0}$. Let $\lambda\in P_\fb^+$ and suppose that  $\lambda=((\lambda^1), (\lambda^2), \ldots)\in \h^*$, with $\lambda^k=(\lambda^k_1,\ldots, \lambda^k_{n_k})\in P_{\fb_{n_k}}^+$, for all $k\in \Z_{>0}$. Assume also that each simple $\fsl(n_k)$-module $L_{\fb_{n_k}}(\lambda^k)$ has one-dimensional weight spaces. For every $k\in \Z_{> 0}$, we choose $\mu^k\in \Supp L_{\fb_{n_k}}(\lambda^k)$, and we let $\mu=((\mu^1), (\mu^2),\ldots)\in \h^*$. Under these assumptions, we can consider the $\fl$-module $L_{\fb_\fl}((\lambda), (\mu))$ defined above, where $\fl:=\bigoplus_{k\in \Z_{>0}} \fsl(n_k)$. Finally, we set $N_k=\sum_{\ell=1}^k n_\ell$, and $\lambda^{(k)}=((\lambda^1),\ldots, (\lambda^k))\in P_{\fb_{N_k}}$, where $\fb_{N_k}:=\fb\cap \fsl(N_k)$, and we let $\p=\fl\oplus \fu$ be a parabolic subalgebra of $\g$ containing $\fb$.  Consider $L_{\fb_\fl}((\lambda), (\mu))$ as $\p$-module with trivial action of $\fu$. 

\begin{lem}
With the above notation the induced module
	\[
\Ind^\g_\p L_{\fb_\fl}((\lambda), (\mu)):= \bU(\g)\otimes_{\bU(\p)} L_{\fb_\fl}((\lambda), (\mu))
	\]
has a unique maximal proper submodule.
\end{lem}
\begin{proof}
Let $L=L_{\fb_\fl}((\lambda), (\mu))$ and $M=\Ind^\g_\p L_{\fb_\fl}((\lambda), (\mu))$. By PBW Theorem, $\Ind^\g_\p L$ admits a vector space decomposition $L\oplus N$, where $N=\bigoplus_{\mu\notin \Supp L} M_\mu$. Let $N'$ be any proper submodule of $M$. Since $L$ is simple we must have that $N'\cap L=0$ and hence $N'\subseteq N$. In particular, the sum of all proper submodules of $M$ is contained in $N$, and the result is proved.
\end{proof}

Let $V_\p(L_{\fb_\fl}((\lambda), (\mu)))$ denote the unique simple quotient of $\Ind^\g_\p L_{\fb_\fl}((\lambda), (\mu))$.

\begin{prop}
$V_\p(L_{\fb_\fl}((\lambda), (\mu)))$ is integrable with finite-dimensional weight spaces.
\end{prop}
\begin{proof}
Let $L=L_{\fb_\fl}((\lambda), (\mu))$, $M=\Ind^\g_\p L_{\fb_\fl}((\lambda), (\mu))$, $V=V_\p(L_{\fb_\fl}((\lambda), (\mu)))$, and $\pi:M\to V$ the  projection defining $V$. Let $\fu^-$ be the opposite subalgebra of $\fu$. Then $\g=\fu^-\oplus \p$, and by PBW Theorem, $M\cong \bU(\fu^-)\otimes L$. Hence, for every $\nu\in \Supp M$ we have 
	\[
M_\nu\cong \bigoplus_{\lambda+\lambda'=\nu}\bU(\fu^-)_{\lambda}\otimes L_{\lambda'}.
	\] 
Since every block of $\fl$ is finite, it is not hard to see that for any $\nu\in \Supp M$, there are finitely many pairs $(\lambda, \lambda')$ for which $\lambda+\lambda'=\nu$. Furthermore, since $\p$ contains a Dynkin Borel subalgebra $\fb=\h\oplus \n$, it follows that $\fu^-\subseteq \n^-$, where $\n^-$ denotes the opposite subalgebra of $\n$. In particular, since $\fb$ is Dynkin, we have that $\dim \bU(\fu^-)_\beta < \infty$ for all $\beta\in \Supp \bU(\fu^-)$. Now, the claim that $\dim M_\nu<\infty$ follows from the fact that $\dim L_{\mu}=1$ for all $\mu\in \Supp L$. To see that $V$ is integrable we notice that for any $g\in \g$ and $v\in V$, there is $n\gg 0$ such that $g\in \fsl(n)$, and 
	\[
v=\pi(\sum_{i=1}^r u_i\otimes v_i)\in \pi\left(\bU(\fsl(n))\otimes_{\p\cap \fsl(n)} \bigotimes_{k=1}^n L_{\fb_{n_k}}(\lambda^k)\right).
	\]
Since $\lambda\in P_\fb^+$, for each $i=1,\ldots, r$, there is $m_i\in\Z_{>0}$ such that $g^{m_i}\cdot \pi(u_i\otimes v_i)=0$. In particular, this shows that $g^m v=0$ for $m\gg 0$. Thus the result follows.
\end{proof}

Consider an infinite subset $A=\{a_1,a_2,\ldots\mid a_i < a_{i+1}\}\subseteq \Z_{>0}$. For any $a_n\in \Z_{>0}$ we have a unique, up to a scalar multiplication, embedding of $\fsl(n)$-modules $\Lambda^{a_n}V_{n}\hookrightarrow \Lambda^{a_{n+1}}V_{n+1}$. We define the $\g$-module $S_A^{\infty} V$ to be the direct limit $\varinjlim_{n}S^{a_n}V_{n}$ (see \cite{GP18} for details). The following claim due to I. Dimitrov \cite{D_talk}: \\ 

\emph{
Any integrable weight module with finite-dimensional weight spaces is isomorphic to either $V_\p(L_{\fb_\fl}((\lambda), (\mu)))$ for some Dynkin Borel subalgebra $\fb\subseteq \g$, some parabolic subalgebra $\p\supseteq \fb$, and $\lambda,\mu \in \h^*$, or to $S_A^\infty V$ for some infinite subset $A=\{a_1, a_2,\ldots\mid a_i< a_{i+1}\}\subseteq \Z_{> 0}$.
}

\section{The annihilator of simple integrable weight modules with finite-dimensional weight spaces}

This section is devoted to compute the annihilators of simple integrable $\g$-modules with finite-dimensional weight spaces. In what follows we let $W_{n}$ denote the Weyl group of $\fsl(n)$ for every $n\in\Z_{>0}$.

\begin{theo}\label{thm:main1}
There exists an exhaustion of $V_\p(L_{\fb_\fl}((\lambda), (\mu)))$ given by the simple $\fsl(N_k)$-modules $L_{\fb_{N_k}}(\lambda^{(k)})$. In particular,  $\Ann V_\p(L_{\fb_\fl}((\lambda), (\mu))) = \Ann L_\fb (\lambda)$.
\end{theo}
\begin{proof}
Notice that $v^k=(\otimes_{i=1}^k v_{\lambda^i})(\otimes_{\ell> k} v_{\mu^\ell})$ is a nonzero $\fb_{N_k}$-highest weight vector of $\h_{N_k}$-weight $\lambda^{(k)}$. Consider the $\fsl(N_k)$-submodule of $V_\p(L_{\fb_\fl}((\lambda), (\mu)))$ generated by $v^k$. Since $V_\p(L_{\fb_\fl}((\lambda), (\mu)))$ is integrable, it follows that $\dim \bU(\fsl(N_k))\cdot v^k < \infty$. Thus $\Supp \bU(\fsl(N_k))\cdot v^k$ is $W_{N_k}$-invariant, and hence $\bU(\fsl(N_k))\cdot v^k\cong L_{\fb_{N_k}}(\lambda^{(k)})$. Notice now that the assignment 
	\[
\varphi_{\mu^k}: L_{\fb_{N_k}}(\lambda^{(k)})\to L_{\fb_{N_{k+1}}}(\lambda^{(k+1)}),\quad (\otimes_{i=1}^k v_{\lambda^i})\mapsto (\otimes_{i=1}^k v_{\lambda^i})\otimes v_{\mu^{k+1}}
	\]
gives a well defined embedding of $\fsl(N_k)$-modules. Since $V_\p(L_{\fb_\fl}((\lambda), (\mu)))$ is simple, the family of maps $\{\varphi_{\mu^k}\}_{k\in \Z_{>0}}$ gives the desired exhaustion. Finally, the modules $L_{\fb_{N_k}}(\lambda^{(k)})$ also give an exhaustion for $L_{\fb}(\lambda)$ (via a different family of embeddings). Finally, the result follows from the fact that if $M$ is a $\varinjlim_n\g(n)=\g$-module such that $M=\varinjlim_n M_n$, where each $M_n$ is a $\g(n)$-module, then $\Ann M = \cap \Ann M_n$.
\end{proof}

\begin{cor}
Let $\im (\lambda)=\{\lambda_i\mid i\in \Z_{>0}\}$. Then $\Ann V_\p(L_{\fb_\fl}((\lambda), (\mu))) \neq 0$ if and only if $|\im (\lambda)|<\infty$.
\end{cor}
\begin{proof}
This follows from Theorem~\ref{thm:main1} and \cite[Theorem~9]{PP16}.
\end{proof}

\begin{cor}\label{cor:iff.hw}
$V_\p(L_{\fb_\fl}((\lambda), (\mu)))$ is a highest weight module with respect to some Borel subalgebra if and only if there is $k_0\in\Z_{>0}$, $\nu^{(k_0)}\in W_{N_{k_0}}\cdot \lambda^{(k_0)}$ such that $(\nu^{(k_0)}, \mu^{k_0+1},\ldots, \mu^{k_0+n})\in W_{N_{k+n}}\cdot \lambda^{(k_0+n)}$ for any $n>0$.
\end{cor}
\begin{proof}
To see this we first recall from Theorem~\ref{thm:main1} that $V_\p(L_{\fb_\fl}((\lambda), (\mu)))\cong \varinjlim_k L_{\fb_{N_k}}(\lambda^{(k)})$. Since all possible highest weights of $L_{\fb_{N_k}}(\lambda^{(k)})$ lie in $W_{N_k}\cdot \lambda^{(k)}$, and every weight $\nu^{(k)}\in \Supp  L_{\fb_{N_k}}(\lambda^{(k)})$ is send to the weight $(\nu^{(k)}, \mu^{k+1})\in \Supp  L_{\fb_{N_{k+1}}}(\lambda^{(k+1)})$, the result follows.
\end{proof}

\begin{lem}\label{lem:ann.bound}
Let $c\in\Z_{>0}$, and $M$ be a simple bounded $\g$-module such that $\dim M_\mu< c$, for all $\mu\in \Supp M$. If $N$ is a simple $\g$-module with $\Ann M = \Ann N$, then $N$ is also bounded, and $\dim N_\mu<c$ for all $\mu\in \Supp N$.
\end{lem}
\begin{proof}
By a direct verification one can see that the proof of \cite[Theorem~4.3]{PS12} also works for the case $\g=\fsl(\infty)$.
\end{proof}

Let $A$ be a semi-infinite subset of $\Z_{>0}$ (i.e. $|A|=\infty$ and $|\Z_{>0}\setminus A|=\infty$). We say that $A$ is compatible with a linear order $\prec$ on $\Z_{>0}$ if $a\in A$, $b\in \Z_{>0}\setminus A$ implies $a\prec b$. For every semi-infinite subset $A\subseteq \Z_{>0}$ we define the element $\varepsilon_A := \sum_{i\in A}\varepsilon_i\in \h^*$. Now we have the following result.

\begin{cor}\label{cor:iff.bound}
$V_\p(L_{\fb_\fl}((\lambda), (\mu)))$ is a bounded weight module if and only if one of the following statements hold:
\begin{enumerate}
\item \label{item1:cor:iff.bound} $(\Z_{>0}, \prec)\cong (\Z_{>0},<)$, and there is $k\in \Z_{>0}$ and a partition $\mu = (\mu_1\geq \mu_2\geq\cdots \geq \mu_k)$ such that $\lambda=\sum_{j=1}^k \mu_j\varepsilon_{i_j}$, where $i_1\prec\cdots \prec i_k$ is the left end of $\prec$;
\item \label{item2:cor:iff.bound} $(\Z_{>0}, \prec)\cong (\Z_{<0},<)$, and there is $k\in \Z_{>0}$ and a partition $\mu = (\mu_1\geq \mu_2\geq\cdots \geq \mu_k)$ such that $\lambda=\sum_{j=1}^k -\mu_j\varepsilon_{i_j}$, where $i_k\prec\cdots \prec i_1$ is the right end of $\prec$;
\item \label{item3:cor:iff.bound} $(\Z_{>0}, \prec)\cong (\Z,<)$, and $\lambda=\varepsilon_A$ for some semi-infinite subset $A\subseteq \Z_{>0}$ which is compatible with $\prec$.
\end{enumerate}
\end{cor}
\begin{proof}
By Theorem~\ref{thm:main1}, we have $\Ann V_\p(L_{\fb_\fl}((\lambda), (\mu))) = \Ann L_\fb(\lambda)$. In particular, it follows from Lemma~\ref{lem:ann.bound} that $V_\p(L_{\fb_\fl}((\lambda), (\mu)))$ is bounded if and only if $L_\fb(\lambda)$ is bounded. Since $\fb$ is a Dynkin Borel subalgebra, it follows from \cite[Theorem~5.1 and Proposition~5.2]{GP18} that $L_\fb(\lambda)$ is a bounded weight module if and only if \eqref{item1:cor:iff.bound}, \eqref{item2:cor:iff.bound} or \eqref{item3:cor:iff.bound} hold.
\end{proof}

In \cite{PP18} a parametrization of all primitive ideals of $\bU(\fsl(\infty))$ is given in terms of quadruples $(r,g,X,Y)$, where $r,g\in \Z_{\geq 0}$ and $X,Y$ are Young diagrams. The primitive ideal associated to $(r,g,X,Y)$ is denoted by $I(r,g,X,Y)$. In \cite{PP18} the authors describe how to find the quadruple $(r,g,X,Y)$ associated to $\Ann L_\fb(\lambda)$, and the parameters $r,g,X,Y$ are given in terms of $\lambda$. We set
	\[
\Ann L_\fb(\lambda):=I(r(\lambda),g(\lambda),X(\lambda),Y(\lambda)).
	\]

The following theorem is the main result of this paper. It uses Dimitrov's claim (see the end of  Section~\ref{subsec:classification.int.mod}).

\begin{theo}\label{thm:main2}
If $M$ is an integrable simple weight module with finite-dimensional weight spaces, then one of the following statements hold:
\begin{enumerate}
\item $M\cong V_\p(L_{\fb_\fl}((\lambda), (\mu)))$ where $\p\supseteq \fb$ is a left-infinite Dynkin Borel subalgebra, and $\Ann M=I(r(\lambda),0,\emptyset,Y(\lambda))$.
\item $M\cong V_\p(L_{\fb_\fl}((\lambda), (\mu)))$ where $\p\supseteq \fb$ is a right-infinite Dynkin Borel subalgebra, and $\Ann M=I(r(\lambda^*),0,X(\lambda^*),\emptyset)$, where $\lambda^*:=(\ldots, -\lambda_3, -\lambda_2,-\lambda_1)\in \h^*$.
\item $M\cong V_\p(L_{\fb_\fl}((\lambda), (\mu)))$ where $\p\supseteq \fb$ is a two-sided Dynkin Borel subalgebra, and $\Ann M=I(r(\lambda),g(\lambda),\emptyset,\emptyset)$.
\item $M\cong S_A^\infty(V)$, for some infinite subset $A=\{a_1, a_2,\ldots\mid a_i< a_{i+1}\}\subseteq \Z_{> 0}$, and $\Ann M=I(1,0,\emptyset,\emptyset)$.
\end{enumerate}
\end{theo}
\begin{proof}
First notice that if $\fb_1$ and $\fb_2$ are conjugate (under $\Aut (\g)$) Borel subalgebras of $\g$, then $\Ann L_{\fb_1}(\nu)=\Ann L_{\fb_2}(\nu)$ for any $\nu\in \h^*$. Now the result follows from Theorem~\ref{thm:main1} along with the description of $\Ann L_\fb(\lambda)$ given in \cite[\S~6]{PP18} for the case where $\fb$ is a Dynkin Borel subalgebra of $\g$.
\end{proof}

Now we can use Corollaries~\ref{cor:iff.hw}~and~\ref{cor:iff.bound} to construct examples of simple integrable weight modules with finite-dimensional weight spaces which are neither highest weight nor bounded. The next example was given in I. Dimitrov's talk \cite{D_talk}.

\begin{example}
Let $\fb$ be the right-infinite Dynkin Borel subalgebra, and consider the element $\lambda = (0,-2,-4,-6,-8,-10,-12,\ldots)\in P_{\fb}^+$. Since $\lambda_i-\lambda_{i+1}=2$ for all $i\in \Z_{>0}$, a way of decomposing $\lambda$ into blocks $\lambda=((\lambda^1), (\lambda^2),\ldots )$ so that each $L_{\fb_{n_k}}(\lambda^k)$ have one-dimensional weight spaces is given by $\lambda=((-0,-2), (-4,-6), (-8,-10),\ldots)$. In particular, we must have $\fl=\bigoplus_{k\in \Z_{>0}}\fsl(2)$ as the reductive component of $\p$, and $L_{\fb\cap \fsl(2)}(\lambda^k)\cong_{\fsl(2)} \fsl(2)$, for all $k> 0$. Notice that
	\[
\Supp L_{\fb\cap \fsl(2)}(\lambda_i, \lambda_{i+1})=\{(\lambda_i, \lambda_{i+1}), (\lambda_i-1, \lambda_{i+1}+1), (\lambda_i-2, \lambda_{i+1}+2)\}
	\]
for each $i\in \Z_{>0}$. Now we can choose $\mu=((-1,-1), (-5,-5), (-9,-9),\ldots )\in \h^*$. Since every weight $\nu^{(k)}\in \Supp  L_{\fb_{N_k}}(\lambda^{(k)})$ is send (via the exhaustion given in Theorem~\ref{thm:main1}) to the weight $(\nu^{(k)}, \mu^{k+1})\in \Supp  L_{\fb_{N_{k+1}}}(\lambda^{(k+1)})\setminus W_{N_{k+1}}\cdot \lambda^{(k+1)}$, we conclude from Corollary~\ref{cor:iff.hw} that $V_\p(L_{\fb_\fl}((\lambda), (\mu)))$ is not a highest weight module with respect to any Borel subalgebra of $\g$. Moreover, it follows from Corollary~\ref{cor:iff.bound} that $L_\fb(\lambda)$ cannot be a bounded module. Thus $V_\p(L_{\fb_\fl}((\lambda), (\mu)))$ is not bounded either.

Notice that if we choose $\eta=((-2,-0), (-6,-4), (-10,-8),\ldots )$, then $V_\p(L_{\fb_\fl}((\lambda), (\eta)))$ is a highest weight module with respect to the opposite Borel of $\fb$. Although it is still not bounded.
\end{example}


\bibliographystyle{alpha}

\bibliography{/Users/lucascalixto/Dropbox/Research/Bib/bibliography}

\end{document}